\documentclass[11pt]{article}
\usepackage{amsmath, amssymb, amsthm}
\usepackage{subfig}
\usepackage[margin=1in]{geometry}
\usepackage{bbm}
\usepackage{graphicx, verbatim}
\usepackage{setspace}
\usepackage{url}
\usepackage{float}

\newtheorem{thm}{Theorem} % Theorem environment
\newtheorem{cor}{Corollary}    % Corollary environment
\newtheorem{lemma}{Lemma}        % Lemma environment
\newtheorem{remark}{Remark}

\title{Strong Convergence  of Euler  Approximations of Stochastic Differential  Equations with Delay under Local Lipschitz Condition}
\author{Chaman Kumar and Sotirios Sabanis\footnote{The authors would like to thank Lukas Szpruch for his useful discussions.} \\ School of Mathematics, \\ University of Edinburgh, \\ Edinburgh, EH9 3JZ, U.K.}

\date{\today}

\begin{document}
\pagenumbering{arabic}
\maketitle
\begin{abstract}
The strong convergence of Euler  approximations of stochastic delay differential equations is proved under general conditions. The assumptions on drift and diffusion coefficients have been relaxed to include polynomial growth and only continuity in the arguments corresponding to delays. Furthermore, the rate of convergence is obtained under one-sided and polynomial Lipschitz conditions. Finally, our findings are demonstrated with the help of numerical simulations.
\newline \newline
\textit{Keywords}: stochastic delay differential equations, Euler approximations, strong convergence, rate of convergence, local Lipschitz condition.
\end{abstract}
%\begin{acknowledgements}
%My acknowledgements.
%\end{acknowledgements}
%\declaration
%\tableofcontents
\section{Introduction}
In modeling many real world phenomena, the future states of the system depend not only on the present state but also on its past state(s). The models based on  stochastic delay differential equations (SDDEs) could be used in such situations. They have found numerous applications in various fields, for example,  in  communications, physics, biology, ecology, economics and finance. One could refer to  \cite{arriojas2007} - %, \cite{christian1996}, \cite{Federico2010}, \cite{Federico2011}, \cite{goldobin2003}, \cite{huber2003},  \cite{mao1997}, \cite{mao2012}, \cite{nairn2011},
\cite{Oksendal2011}  and references therein.

In many applications, one finds SDDEs for which   solutions can not be obtained explicitly and hence this necessitates the development of numerical schemes. Over the past few years, many authors have shown their interest in studying strong convergence properties of numerical schemes of SDDEs, to name a few, \cite{baker2000} - %, \cite{gs2012},  \cite{hu2004}, \cite{kuchler2000},   \cite{mao2003} and
\cite{maosabanis}.   Recently there has been a growing interest in weak convergence of numerical schemes as well, see for example, \cite{buckwar2008}, \cite{buckwar2005} and \cite{kuchler2002}. We also mention  a paper by \cite{kloeden2007} on pathwise approximations and its relation to strong convergence. For the purpose of this article, we study strong convergence of Euler  approximations of SDDEs due to its wide applicability. We exploit the fact that SDDEs can be regarded as  special cases of  stochastic differential equations (SDEs) with random coefficients. To this end,  first we establish strong convergence results of  Euler  schemes of SDEs with random coefficients under mild conditions. This, in turn, allows us to prove the convergence in the case of SDDEs under more relaxed conditions than those existing in the literature. More precisely, to the best of our knowledge, strong convergence has been proved so far by assuming linear growth and local Lipschitz conditions in both variables corresponding to delay and non-delay arguments, see for example \cite{maosabanis}.  Whereas, we relax these conditions by assuming polynomial growth and continuity in the argument corresponding to delays. Moreover no smoothness condition on the initial data is assumed.  Therefore, the setting in this article is more general than those present in the literature, see for example, \cite{baker2000},  \cite{kuchler2000}, \cite{mao2003},  \cite{maosabanis} and \cite{kloeden2007}. In addition, it is shown here that the rate of convergence is one-fourth when the drift coefficient satisfies one-sided Lipschitz and polynomial Lipschitz in the non-delay and delay arguments respectively. This result is in agreement with the findings of \cite{gyongy1998}.

We conclude this section with the introduction of some basic notation. The norm of a vector $x \in \mathbb R^d$ and the Hilbert-Schmidt norm of a matrix $A\in \mathbb R^{d\times m}$ are respectively denoted by $|x|$ and $|A|$. The transpose of a vector $x \in \mathbb R^d$ is denoted by $x^{T}$ and the inner product of two vectors $x, y \in \mathbb R^d$ is denoted by $\langle x, y \rangle := x^{T}y$. The integer part of a real number $x$ is denoted by $[x]$.  We denote the indicator function of set $A$ by $\mathbbm {1}_A$. Moreover, $\mathcal L^p=\mathcal L^p(\Omega, \mathcal F, \mathbb P)$ denotes the space of random variables $X$ with a norm $\|X\|_p:=\big(\mathbb E\big[|X|^p\big]\big)^{1/p} < \infty$ for $p>0$. Finally, $\mathcal P$ denotes the predictable $\sigma$-algebra on $\mathbb R_+\times \Omega$ and $\mathcal B(V)$, the $\sigma$-algebra of Borel sets of topological spaces $V$.

\section{Main Result}
Let $(\Omega, \{\mathcal F_t\}_{\{t \geq 0\}}, \mathcal F, \mathbb P)$  be a filtered probability space satisfying the usual conditions, i.e. the filtration is increasing and right continuous. Let $\{W(t)\}_{\{t \geq 0\}}$ be an $m$-dimensional Wiener martingale. Furthermore, it is assumed that $\beta(t, y_1, \ldots, y_k, x)$ and $\alpha(t, y_1, \ldots, y_k, x)$ are $\mathcal B(\mathbb R_+) \otimes \mathcal B(\mathbb R^{d\times k}) \otimes \mathcal B(\mathbb R^d)$-measurable functions and take values in $\mathbb R^d$ and $\mathbb R^{d \times m}$ respectively. For a fixed $T>0$, let the stochastic delay differential equation (SDDE) on $(\Omega, \{\mathcal F_t\}_{\{t \geq 0\}}, \mathcal F, \mathbb P)$ be defined  as follows,
\begin{eqnarray}
\begin{array}{lcll}
dX(t) &=& \beta(t, Y(t), X(t) )dt+\alpha(t, Y(t), X(t) )dW(t), & t \in [0,T],
\\
X(t)& =& \xi(t), & t \in [-H, 0 ], \label{sdde}
\end{array}
\end{eqnarray}
where $\{\xi(t): -H \leq t \leq 0\} \in \mathcal C_{\mathcal F_0}^{b}([-H,0];\mathbb R^d)$ for  some $H >0$ and $Y(t):=(X(\delta_1(t)), \ldots, X(\delta_k(t)))$.  The delay parameters $\delta_1(t), \ldots, \delta_k(t)$ are increasing functions of $t$ and satisfy $-H \leq \delta_j(t) \leq [\frac{t}{\tau}]\tau$ for some $\tau>0$ and $j=1,\ldots, k$.

Further, for every  $n \geq 1$ and $t \in [0,T]$, the Euler  scheme of SDDE \eqref{sdde} is given by
\begin{eqnarray}
\begin{array}{lcll}
dX_n(t) &=& \beta(t, Y_n(t), X_n(\kappa_n(t))) dt + \alpha(t, Y_n(t), X_n(\kappa_n(t))) dW(t), & t \in [0, T],
\\
X_n(t) &=& \xi(t), & t \in [-H, 0], \label{sddeem}
\end{array}
\end{eqnarray}
where $Y_n(t):=(X_n(\delta_1(t)), \ldots, X_n(\delta_k(t)))$ and $\kappa_n$ is defined by
\begin{eqnarray}
\kappa_n(t):=\frac{[n(t-t_0)]}{n}+t_0 \label{kappa}
\end{eqnarray}
with the observation that for equation \eqref{sddeem} one takes $t_0=0$ in \eqref{kappa}. We note that two popular cases of delay viz. $\delta_i(t)=t-\tau$ and $\delta_i(t)=\Big[\frac{t}{\tau}\Big]\tau$ can be addressed by our findings. This type of delay parameters have been extensively discussed and found wide applications in the literature. The reader could consult the following, for example, \cite{arriojas2007} - %, \cite{christian1996},  \cite{Federico2010}, \cite{Federico2011}, \cite{goldobin2003},
\cite{huber2003}, \cite{Oksendal2011} and  references therein.
\begin{remark}
We note that the Euler scheme \eqref{sddeem} defines approximations to SDDEs in an explicit way without a discretization  of the delay terms. From our main theorem and corollaries on convergence of scheme \eqref{sddeem}, one could easily obtain results on convergence of this scheme  with discretized delay terms. Therefore for matters of notational simplicity, we choose the former approach.
\end{remark}
Let  $y := (y_1, \ldots, y_k)$. We make the following assumptions for our result.
\newline
\textbf{C-1.} There exist constants $G >0$ and $l>0$ such that for any $t \in [0,T]$,
%\begin{eqnarray}
$$|\beta(t, y, x)|+|\alpha(t, y, x)| \leq G(1+|y|^l+|x|)$$ %\notag
%\end{eqnarray}
for all $x  \in \mathbb R^d$ and $y \in \mathbb R^{d \times k}$.
\newline
\textbf{C-2.}  For every $R>0$, there exists a constant $K_R >0 $ such that for any $t \in [0,T]$,
\begin{eqnarray}
\langle x-z,\beta(t, y, x)-\beta(t, y, z) \rangle \vee |\alpha(t, y, x)-\alpha(t, y, z)|^2 \leq K_R |x-z|^2 \notag
\end{eqnarray}
whenever $|x|, |y|, |z| < R$.
\newline
\textbf{C-3.} The functions $\beta(t,y,x)$ and $\alpha(t,y,x)$ are continuous in $y$ uniformly in $x$ from compacts, i.e. for every $R>0$ and $t \in [0,T]$,
\begin{eqnarray}
\sup_{|x| \leq R}\{|\beta(t, y, x)-\beta(t, y^{'}, x)|+|\alpha(t, y, x)-\alpha(t, y^{'}, x)|\} \rightarrow 0 \quad \mbox{as} \quad y^{'} \rightarrow y. \notag
\end{eqnarray}

The conditions C-1 and C-2 are sufficient for existence and uniqueness of solution of SDDE \eqref{sdde} and the Euler scheme \eqref{sddeem}(see \cite{gs2012}, \cite{renesse2010}). We state the main result of this paper in the following theorem.
\begin{thm}
\label{mainthm}
Suppose C-1 to C-3 hold, then the Euler  scheme \eqref{sddeem} converges to the true solution of SDDE \eqref{sdde} in $\mathcal L^p$-sense, i.e.
\begin{eqnarray}
\lim_{n \rightarrow \infty}\mathbb E \Bigg[ \sup_{0 \leq t \leq T}|X(t)-X_n(t)|^p\Bigg]=0 \notag
\end{eqnarray}
for all $p >0 $.
\end{thm}
\begin{remark}
We note that we can assume, without loss of generality, that  $T$ is a  multiple of $\tau$. If not, then SDDE \eqref{sdde} and its EM scheme \eqref{sddeem} are defined for $T' > T$ with $T'=N\tau$, where $N$ is a positive integer.  The results given in this report are then recovered for the original SDDE \eqref{sdde} by choosing drift and diffusion parameters as $\beta 1_{\{t\leq T\}}$ and $\alpha 1_{\{t \leq T\}}$.
\end{remark}

First, we develop requisite theory for SDEs with random coefficients in the following section and then prove Theorem \ref{mainthm} in Section \ref{pruf}.

\section{SDEs with Random Coefficients}
Let $b(t,x)$ and $\sigma(t,x)$ be $\mathcal P \otimes \mathcal{B}(\mathbb R^d)$-measurable functions which take values in $\mathbb R^d$ and $\mathbb R^{d \times m}$ respectively. Suppose $0 \leq t_0 < t_1 \leq T$, then let us consider an SDE with random coefficients given by
\begin{eqnarray}
dX(t)=b(t, X(t))dt+\sigma(t,X(t))dW(t), \, \, \, \, \forall  \, \, \, t\in [t_0, t_1], \label{sderc}
\end{eqnarray}
with initial value $X(t_0)$ which is an almost surely finite $\mathcal F_{t_0}$-measurable random variable.

For every $n \geq 1$, let $b_n(t,x)$ and $\sigma_n(t,x)$ be $\mathcal P \otimes \mathcal B(\mathbb R^d)$-measurable functions which take values in $\mathbb R^d$ and $\mathbb R^{d \times m}$ respectively. Then for any $t \in [t_0,t_1]$, we define the following Euler  scheme corresponding to the SDE \eqref{sderc} as,
\begin{eqnarray}
dX_n(t)=b_n(t,X_n(\kappa_n(t)))dt+\sigma_n(t,X_n(\kappa_n(t)))dW(t) \label{em}
\end{eqnarray}
with initial value $X_n(t_0)$, which is an $\mathcal F_{t_{0}}$-measurable, almost surely finite, random variable and $\kappa_n$ as defined in \eqref{kappa}.
\newline
We make the following assumptions.
\newline
\textbf{A-1.} There exists a constant $c>0$ and non-negative random variables $\{M_n\}_{n \geq 1}$ and $M$ with $\mathbb E\big[M^p\big] \vee \sup_{n \geq 1} \mathbb E\big[M_n^p\big] < N$ for every $p >  0$ and some $N:=N(p)>0$  such that, almost surely,
\begin{eqnarray}
&& |b(t,x)| + |\sigma(t,x)| \leq c(M+|x|) \notag
\\
&& |b_n(t,x)| + |\sigma_n(t,x)| \leq c(M_n+|x|) \notag
\end{eqnarray}
for any $t \in [0,T]$ and $x \in \mathbb R^d$.
%
%Furthermore, we assume that for every $n \geq 1$, there exists a constant $c>0$ and a non-negative random variable $M_n$  with $E\big[|M_n|^p\big] < N$ for every $p >0$ and some $N:=N(p)>0$ such that, almost surely,
%\begin{eqnarray}
%|b_n(t,x)| + |\sigma_n(t,x)| \leq c(|M_n|+|x|) \notag
%\end{eqnarray}
%for any $t \in [0,T]$  and $x \in \mathbb R^d$.
\newline
\textbf{A-2.} For every $R>0$, there exists an $\mathcal F_{t_0}$-measurable random variable $C_R$ such that, almost surely, for any $t \in [t_0,t_1]$,
\begin{eqnarray}
\langle x-y, b(t,x)-b(t,y) \rangle \vee |\sigma(t,x)-\sigma(t,y)|^2 \leq C_R |x-y|^2  \notag
\end{eqnarray}
 for all $|x|,  |y| \leq R$. Moreover, there exists a function $f: \mathbb R_+ \rightarrow \mathbb R_+$ such that

$$\lim_{R \rightarrow \infty}\mathbb P (C_R > f(R)) =0.$$
\newline
\textbf{A-3.} For every $R>0$ and $p>0$,
\begin{eqnarray}
\lim_{n \rightarrow \infty}\mathbb E\Bigg[\int_{t_0}^{t_1}\sup_{|x| \leq R}\{|b_n(t,x)-b(t,x)|^p + |\sigma_n(t,x)-\sigma(t,x)|^p\} dt \Bigg]=0.   \notag
\end{eqnarray}
\textbf{A-4.} The initial values of SDE \eqref{sderc} and its Euler  scheme \eqref{em} satisfy
\begin{eqnarray}
\lim_{n \rightarrow \infty}\mathbb E\big[|X(t_0)-X_n(t_0)|^p\big]= 0 \notag
\end{eqnarray}
for every $p >0$.
\newline
\textbf{A-5.} The initial values of  SDE \eqref{sderc} and its Euler scheme \eqref{em} have bounded $p$-th moments, i.e. there exists $L:=L(p)>0$ such that
\begin{eqnarray}
\mathbb E[|X(t_0)|^p]  \, \, \vee \, \,  \sup_{n \geq 1}\mathbb E[|X_n(t_0)|^p] < L \notag
\end{eqnarray}
for every $p >0$.

We note that assumptions A-1 and A-2 are sufficient for existence and uniqueness of solution of SDE \eqref{sderc} (see \cite{gynongykrylov} and \cite{krylov1990}). Before proving the main result (Theorem \ref{**}) of this section, we establish some lemmas.
\begin{lemma}
\label{bbnn}
 Suppose that  A-1 and A-5 hold, then for some $K:=K(p, T, L)>0$,
\begin{eqnarray}
&& \mathbb E \bigg[ \sup_{t_0\leq t \leq t_1}|X(t)|^p\bigg] <K \notag
\\
\mbox{and} \, \,&& \sup_{n \ge 1}\mathbb E \bigg[ \sup_{t_0\leq t \leq t_1}|X_n(t)|^p\bigg] < K \label{hash}
\end{eqnarray}
for every $p >0$.
\end{lemma}
\begin{proof}
First one  chooses any $p > 2$. Then  for the true solution $\{X(t)\}_{t \in [t_0, t_1]}$, by using H\"older's inequality, one obtains,
\begin{eqnarray}
|X(t)|^p &\leq & 3^{p-1}\Bigg\{ |X(t_0)|^p+|t-t_0|^{p-1} \int_{t_0}^{t} |b(s, X(s))|^p ds + \Bigg| \int_{t_0}^{t}\sigma(s,X(s))dW(s) \Bigg|^p \Bigg\}. \notag
\end{eqnarray}
Thus on taking supremum over $[t_0, u] $ for some $u \in [t_0, t_1]$ and by applying Burkholder-Davis-Gundy inequality with constant $\bar c:=\bar c(p)>0$, one  observes that
\begin{eqnarray}
\mathbb E\bigg[\sup_{t_0 \leq t \leq u}|X(t)|^p \bigg] \leq  3^{p-1}\bigg\{ \mathbb E[|X(t_0)|^p]+ T^{p-1} \mathbb E \bigg[\int_{t_0}^{u} |b(s, X(s))|^p ds\bigg] + \bar c \mathbb E\bigg[ \int_{t_0}^{u}|\sigma(s,X(s))|^p ds \bigg] \bigg\}. \notag
\end{eqnarray}
Now A-1 and A-5 give
\begin{eqnarray}
\mathbb E\bigg[\sup_{t_0 \leq t \leq u}|X(t)|^{p} \bigg] &\leq &  2^{p-1}\bigg\{ \mathbb E[|X(t_0)|^p]+ \big[T^{p-1}+\bar c\big] c^p \mathbb E \bigg[\int_{t_0}^{u} \big\{|M| +|X(s)|\big\}^{p} ds\bigg]  \bigg\} \notag
\\
&  \leq &  3^{p-1}\bigg\{ \mathbb E[|X(t_0)|^p] +  2^{p-1}\big[T^{p-1}+\bar c\big]c^p T N \notag
\\
&& \qquad + 2^{p-1}\big[T^{p-1}+\bar c\big]c^p  \int_{t_0}^{u} \mathbb E \bigg[\sup_{t_0 \leq r\leq s}|X(r)|^p \bigg] ds  \bigg\} \notag
\end{eqnarray}
which on the application of Gronwall's inequality yields
\begin{eqnarray}
\mathbb E\bigg[\sup_{t_0 \leq t \leq t_1}|X(t)|^p \bigg] < 3^{p-1}\big\{ \mathbb E[|X(t_0)|^p] +  2^{p-1}\big[T^{p-1}+\bar c\big]c^p T N \big\}e^{6^{p-1}(T^{p-1}+\bar c)c^pT}. \notag
\end{eqnarray}
For any $p>2$, we adopt similar arguments for $\{X_n(t)\}_{\{t \in [t_0, t_1]\}}$ in order to prove \eqref{hash} with the remark that A-5 guarantees that $K$ does not depend on $n$. Then one observes that the application  of H\"older's inequality completes the proof.
\end{proof}
For every $R>0$ and $n \geq 1$, let us define the stopping times,
\begin{eqnarray}
\tau_R:=\inf\{t \geq t_0: |X(t)| \geq R\}, \quad \sigma_{nR}:=\inf\{t \geq t_0:|X_n(t)| \geq R\} \quad  \text{and} \quad  \nu_{nR}:=\tau_R  \wedge \sigma_{nR}, \label{nmc}
\end{eqnarray}
where $\inf \emptyset = \infty.$

We  prove below a very useful lemma of this article as it is used in the penultimate section in order to recover the rate of convergence.
\begin{lemma}
\label{l4}
Let us assume that A-1 and A-5 hold. Then
\begin{eqnarray}
\label{lem2}
\lim_{n \rightarrow \infty}\mathbb E \Bigg[ \int_{t_0}^{t_1 }|X_n(s \wedge \nu_{nR})-X_n(\kappa_n(s \wedge \nu_{nR}))|^p ds \Bigg] =0 \notag
\end{eqnarray}
for every $p > 0$.
\end{lemma}
\begin{proof}
In order to prove the result, one chooses $p \geq 2$.  Then one immediately writes,
\begin{eqnarray}
&&\mathbb E \bigg[ \int_{t_0}^{t_1}|X_n(s \wedge \nu_{nR})-X_n(\kappa_n(s \wedge \nu_{nR}))|^p ds \bigg] = \mathbb E\bigg[ \int_{t_0}^{t_1}\bigg| \int_{\kappa_n(s \wedge \nu_{nR})}^{s \wedge \nu_{nR}}b_n(r, X_n(\kappa_n(r)))dr \notag
\\
&& \qquad \qquad \qquad  \qquad  +\int_{\kappa_n(s \wedge \nu_{nR})}^{s \wedge \nu_{nR}}\sigma_n(r, X_n(\kappa_n(r)))dW(r) \bigg|^p ds \bigg] \notag
\end{eqnarray}
which implies on applying H\"older's inequality
\begin{eqnarray}
\lefteqn{ \mathbb E \bigg[ \int_{t_0}^{t_1}|X_n(s \wedge \nu_{nR})-X_n(\kappa_n(s \wedge \nu_{nR}))|^p ds \bigg]  } \notag
\\
&& \leq 2^{p-1} \mathbb E\bigg[ \int_{t_0}^{t_1} |s \wedge \nu_{nR}-\kappa_n(s \wedge \nu_{nR})|^{p-1} \bigg(\int_{\kappa_n(s \wedge \nu_{nR})}^{s \wedge \nu_{nR}}|b_n(r, X_n(\kappa_n(r)))|^p dr\bigg)ds\bigg] \notag
\\
&& \qquad \qquad +2^{p-1} \int_{t_0}^{t_1} \mathbb E\bigg[\bigg |\int_{\kappa_n(s \wedge \nu_{nR})}^{s \wedge \nu_{nR}}\sigma_n(r, X_n(\kappa_n(r)))dW(r) \bigg|^p\bigg]  ds  \notag
\\
&&  \leq \Big(\frac{2}{n}\Big)^{p-1} \int_{t_0}^{t_1}  \mathbb E\bigg[\int_{\kappa_n(s \wedge \nu_{nR})}^{s \wedge \nu_{nR}} |b_n(r, X_n(\kappa_n(r)))|^p dr\bigg]ds  \notag
\\
&& \qquad \qquad + 2^{p-1} \int_{t_0}^{t_1} \mathbb E\bigg[\bigg |\int_{\kappa_n(s \wedge \nu_{nR})}^{s \wedge \nu_{nR}}\sigma_n(r, X_n(\kappa_n(r)))dW(r) \bigg|^p\bigg]  ds.  \label{g1}
\end{eqnarray}
One finds bounds for the integrand of the first term of \eqref{g1} as follows,
\begin{eqnarray}
\lefteqn{\mathbb E\bigg[\int_{\kappa_n(s \wedge \nu_{nR})}^{s \wedge \nu_{nR}} |b_n(r, X_n(\kappa_n(r)))|^p dr\bigg]  \leq  c^p \mathbb E\bigg[\int_{\kappa_n(s \wedge \nu_{nR})}^{s \wedge \nu_{nR}}(|M_n|+|X_n(\kappa_n(r))|)^p dr \bigg]}  \notag
\\
& \leq & c^p 2^{p-1} \mathbb E\bigg[\int_{\kappa_n(s \wedge \nu_{nR})}^{s \wedge \nu_{nR}}(|M_n|^p+|X_n(\kappa_n(r))|^p) dr\bigg]  \notag
\\
& \leq & c^p 2^{p-1} \mathbb E\bigg[\Big(|M_n|^p+\sup_{t_0 \leq r \leq t_1}|X_n(r)|^p\Big)  |s \wedge \nu_{nR}-\kappa_n(s \wedge \nu_{nR})|\bigg] \notag
\\
& \leq & \frac{c^p 2^{p-1}}{n} \Big(\mathbb E\big[|M_n|^p\big]+\mathbb E\bigg[\sup_{t_0 \leq r \leq t_1}|X_n(r)|^p\bigg]\bigg)   \notag
\\
& \leq & \frac{c^p 2^{p-1}}{n} (K+N). \qquad \mbox{(using Lemma \ref{bbnn})}\label{g2}
\end{eqnarray}
For the second term of \eqref{g1}, one writes,
\begin{eqnarray}
&& \hspace{-5mm} \mathbb E\bigg[\bigg |\int_{\kappa_n(s \wedge \nu_{nR})}^{s \wedge \nu_{nR}}\sigma_n(r, X_n(\kappa_n(r)))dW(r) \bigg|^p\bigg] = \mathbb E\bigg[\bigg |\int_{\kappa_n(s)}^{s} \mathbbm{1}_{[\kappa_n(s \wedge \nu_{nR}), s \wedge \nu_{nR}]}\sigma_n(r, X_n(\kappa_n(r)))dW(r) \bigg|^p\bigg] \notag
\\
 && \hspace{10mm} \leq \bigg(\frac{p(p-1)}{2} \bigg)^{p/2}|s-\kappa_n(s)|^{(p-2)/2}\mathbb E\bigg[\int_{\kappa_n(s)}^{s} \mathbbm{1}_{[\kappa_n(s \wedge \nu_{nR}), s \wedge \nu_{nR}]} |\sigma_n(r, X_n(\kappa_n(r)))|^p dr\bigg]  \notag
\end{eqnarray}
which on the application of A-1 and Lemma \ref{bbnn} yields
\begin{eqnarray}
 && \mathbb E\bigg[\bigg |\int_{\kappa_n(s \wedge \nu_{nR})}^{s \wedge \nu_{nR}}\sigma_n(r, X_n(\kappa_n(r)))dW(r) \bigg|^p\bigg]  \leq \bigg(\frac{p(p-1)}{2} \bigg)^{p/2} \frac{1}{n^{p/2}} c^p 2^{p-1}  (K+N). \label{g3}
\end{eqnarray}
Substituting \eqref{g2} and \eqref{g3} in \eqref{g1} yields
\begin{eqnarray}
\mathbb E \bigg[ \int_{t_0}^{t_1}|X_n(s \wedge \nu_{nR})-X_n(\kappa_n(s \wedge \nu_{nR}))|^p ds \bigg] &\leq & \frac{T(K+N)4^{p-1}c^p}{n^{p/2}} \bigg\{\frac{1}{n^{p/2}} + \bigg(\frac{p(p-1)}{2} \bigg)^{p/2} \bigg\} \notag
\\
&\leq & \mathcal O(n^{-p/2}). \label{four4}
\end{eqnarray}
Then by letting $n \rightarrow \infty$ on both sides, one obtains
\begin{eqnarray}
&&\lim_{n \rightarrow \infty} \mathbb E \Bigg[ \int_{t_0}^{t_1}|X_n(s \wedge \nu_{nR})-X_n(\kappa_n(s \wedge \nu_{nR}))|^p ds \Bigg] =0 \notag
\end{eqnarray}
for all $p \geq 2$. Then one observes that the application of H\"older's and Jensen's inequalities completes the proof.
\end{proof}
In order to prove the following theorem, a similar approach to the one adopted by \cite{highammaistaurt} for SDEs is followed here, but in a more general context.
%\begin{comment}
%**********************************************
%\begin{lemma}
%\label{uv}
%For $u, v, \eta , q, p >0$ and $q>p$, the following holds
%\begin{eqnarray}
%uv \leq \frac{\eta p}{q} u^{q/p}+\frac{q-p}{q}v^{q/(q-p)} \frac{1}{\eta^{p/(q-p)}}. \notag
%\end{eqnarray}
%\end{lemma}
%\begin{proof}
%From the Young's inequality, we know that for any $a, b >0$ and $p, q >1$
%\begin{eqnarray}
%ab \leq \frac{a^p}{p}+\frac{b^q}{q} \notag
%\end{eqnarray}
%where $\frac{1}{p}+\frac{1}{q}=1$. Then by taking $p=\mu$, $q=\nu$, $a=\eta^{1/\mu} u$ and $b=\frac{v}{\eta^{1/\mu}}$, we get
%\begin{eqnarray}
%\eta^{1/\mu} u  \frac{v}{\eta^{1/\mu}} \leq \frac{\eta u^\mu}{\mu}+\frac{v^\nu}{\nu \eta^{\nu/\mu}} \quad \Rightarrow uv \leq \frac{\eta u^\mu}{\mu}+\frac{v^\nu}{\nu \eta^{\nu/\mu}} \notag
%\end{eqnarray}
%where $\frac{1}{\mu}+\frac{1}{\nu}=1$. Further for $q>p$, by taking $\mu=\frac{q}{p}$ and $\nu=\frac{q}{q-p}$, we prove the inequality.
%\end{proof}
%*******************************************************
%\end{comment}
\begin{thm}
\label{**}
Let A-1 to A-5 hold and  suppose there  exists a unique solution $\{X_n(t)\}_{\{t_0 \leq t \leq t_1\}}$ of the Euler  scheme \eqref{em}. Then, the Euler  scheme \eqref{em} converges to the SDE  \eqref{sderc} in $\mathcal L^p$-sense, i.e.
\begin{eqnarray}
\lim_{n \rightarrow \infty} \mathbb E\Bigg[ \sup_{t_0 \leq t \leq t_1}|X(t)-X_n(t)|^p \Bigg]=0 \label{uu}
\end{eqnarray}
for all $p > 0$.
\end{thm}
\begin{proof} For every $R>0$ and $n \geq 1$, we consider   stopping times as defined in \eqref{nmc}.
%\begin{comment}
%*******************************************************
%\begin{eqnarray}
%\tau_R:=\inf\{t \geq t_0: |X(t)| \geq R\}, \qquad \sigma_{nR}:=\inf\{t \geq t_0:|X_n(t)| \geq R\} \qquad  \text{and} \qquad  \nu_{nR}:=\tau_R  \wedge \sigma_{nR}, \notag
%\end{eqnarray}
%where $\inf \emptyset = \infty.$
%**************************
%\end{comment}
First we fix $p>4$. Then
\begin{eqnarray}
\mathbb E\Bigg[\sup_{t_0 \leq t \leq t_1}|X(t)-X_n(t)|^p \Bigg] & \leq & \mathbb E \Bigg[ \sup_{t_0 \leq t \leq t_1}|X(t)-X_n(t)|^p \mathbbm{1}_{\{\tau_R \leq t_1 \, \, \text{or} \, \, \sigma_{nR} \leq t_1 \, \, \text{or} \, \, C_R >f(R) \}}\Bigg]  \notag
\\
&& \hspace{-20mm}+ \mathbb E \Bigg[ \sup_{t_0 \leq t \leq t_1}|X(t \wedge \nu_{nR})-X_n(t \wedge \nu_{nR)}|^p \mathbbm{1}_{\{C_R \leq f(R)\}}\Bigg]. \label{one}
\end{eqnarray}
To estimate the first term of \eqref{one}, one can use Young's inequality for $q > p $ $(1/p+1/q=1)$ and $\eta>0$ to obtain the following.
\begin{eqnarray}
\lefteqn{\mathbb E \Bigg[ \sup_{t_0 \leq t \leq t_1}|X(t)-X_n(t)|^p \mathbbm{1}_{\{\tau_R \leq t_1 \, \, \text{or} \, \, \sigma_{nR} \leq t_1 \, \, \text{or} \, \, C_R > f(R) \}}\Bigg]   \leq   \frac{ \eta p }{q}\mathbb E \Bigg[ \sup_{t_0 \leq t \leq t_1}|X(t)-X_n(t)|^q \Bigg]}\notag
\\
&&\qquad \qquad \qquad \qquad \qquad \qquad +\frac{q-p}{q \eta^{p/(q-p)}} \mathbb P(\tau_R \leq t_1 \, \text{or} \, \sigma_{nR} \leq t_1 \, \, \text{or} \, \, C_R > f(R)  ) \notag
\\
& \leq & \frac{ \eta p}{q} 2^{q-1} \Bigg\{\mathbb E \Bigg[ \sup_{t_0 \leq t \leq t_1} |X(t)|^q\Bigg]+ \mathbb E \Bigg[ \sup_{t_0 \leq t \leq t_1} |X_n(t)|^q\Bigg]\Bigg\} \notag
\\
 && \qquad \qquad  \qquad  \qquad  \qquad +\frac{q-p}{q \eta^{p/(q-p)}} \Big \{\mathbb P(\tau_R \leq t_1) + \mathbb P(\sigma_{nR} \leq t_1)+ \mathbb P(C_R > f(R))\Big\} \notag
\\
&  \leq & \frac{ \eta p}{q} 2^{q} K   + \frac{q-p}{q \eta^{p/(q-p)}} \Bigg \{\mathbb E\Bigg[\frac{|X(\tau_R)|^p}{R^p}\Bigg] + \mathbb E\Bigg[\frac{|X_n(\sigma_{nR})|^p}{R^p}\Bigg]+\mathbb P(C_R > f(R))\Bigg\}, \notag
\end{eqnarray}
which becomes
\begin{eqnarray}
&& \mathbb E \Bigg[ \sup_{t_0 \leq t \leq t_1}|X(t)-X_n(t)|^p \mathbbm{1}_{\{\tau_R \leq t_1 \, \, \text{or} \, \, \sigma_{nR} \leq t_1 \, \, \text{or} \, \, C_R > f(R) \}}\Bigg]  \notag
\\
&& \hspace{35mm} \leq  \frac{ \eta p}{q} 2^{q} K  + \frac{q-p}{q \eta^{p/(q-p)}} \Bigg\{\frac{2K}{R^p}+\mathbb P(C_R > f(R))\Bigg\} . \label{two2}
\end{eqnarray}
Further for estimating the second term of \eqref{one}, let $e_n(s):=X(s)-X_n(s)$ for $s \in [t_0, t_1 ]$, then one could write
\begin{eqnarray}
de_n(s)%&=& dX(s)-dX_n(s) \notag
%\\
%&=& \{b(s, X(s))-b_n(s,X_n(\kappa_n(s)))\} ds + \{\sigma(s,X(s))-\sigma_n(s,X_n(\kappa_n(s)))\}dW(s)  \notag
%\\
&=& \bar b_n(s)ds+ \bar \sigma_n(s) dW(s) \notag
\end{eqnarray}
where $\bar b_n(s):= b(s, X(s))-b_n(s,X_n(\kappa_n(s))) $ and $\bar \sigma_n(s):=\sigma(s,X(s))-\sigma_n(s,X_n(\kappa_n(s)))$. Thus, on the application of It$\hat{\text{o}}$'s formula, one obtains
\begin{eqnarray}
 d|e_n(s)|^2 %&=&2 \langle e_n(s), de_n(s) \rangle + | de_n(s) |^2 \notag
%\\
%&=& 2 \langle e_n(s), \bar b_n(s) \rangle ds + 2 \langle e_n(s), \bar \sigma_n(s) dW(s) \rangle + |\bar \sigma_n(s)|^2 ds \notag
%\\
&=& 2 \langle e_n(s), \bar b_n(s) \rangle ds + \sum_{i=1}^{m}|\bar \sigma^i_n(s)|^2 ds + 2 \sum_{i=1}^{m} \langle e_n(s), \bar \sigma^i_n(s) \rangle  dW^i(s)  \notag
\end{eqnarray}
which on integrating over $s \in [t_0, t \wedge \nu_{nR}]$ for some $t \in [t_0, t_1]$ yields
\begin{eqnarray}
 |e_n(t \wedge \nu_{nR})|^2 &=&|e_n(t_0)|^2 + 2 \int_{t_0}^{t \wedge \nu_{nR}}\langle e_n(s), \bar b_n(s) \rangle ds  + \sum_{i=1}^{m} \int_{t_0}^{t \wedge \nu_{nR}} |\bar \sigma^i_n(s)|^2 ds \notag
\\
&& \hspace{10mm} + 2 \sum_{i=1}^{m} \int_{t_0}^{t \wedge \nu_{nR}} \langle e_n(s), \bar \sigma^i_n(s)  \rangle dW^i(s) \notag
\end{eqnarray}
where $\sigma^i_n$ and $W^i$ denote the $i$-th column of $(d \times m)$-matrix $\sigma$ and the $i$-th element of Wiener $(m \times 1)$-vector $W$ respectively for all $i \in \{1, \ldots, m\}$. Further, after raising the power to $p/2$ for some $p>4$ and applying H\"older's and Burkholder-Davis-Gundy  inequalities, one obtains the following,
\begin{eqnarray}
&&  \mathbb E \Bigg[\sup_{t_0 \leq t \leq u}|e_n(t \wedge \nu_{nR})|^p \mathbbm 1_{\{C_R \leq f(R)\}} \Bigg] \leq 2^{p-2}\mathbb E\big[|e_n(t_0)|^p\big] \notag
\\
&& \hspace{15mm} + 2(8T)^{\frac{p-2}{2}} \mathbb E\Bigg[\mathbbm 1_{\{C_R \leq f(R)\}} \int_{t_0}^{u \wedge \nu_{nR}}|\langle e_n(s), \bar b_n(s) \rangle |^{\frac{p}{2}} ds \Bigg]  \notag
\\
&& \hspace{25mm} + (4mT)^{\frac{p-2}{2}} \sum_{i=1}^{m} \mathbb E \Bigg[\mathbbm 1_{\{C_R \leq f(R)\}}\int_{t_0}^{u \wedge \nu_{nR}} |\bar \sigma^i_n(s)|^p ds \Bigg] \notag
\\
&& \hspace{35mm} + 2^{\frac{3p-4}{2}} m^{\frac{p-2}{2}} T^{\frac{p-4}{4}}\bar {\bar c } \sum_{i=1}^{m} \mathbb E \Bigg[\mathbbm 1_{\{C_R \leq f(R)\}} \int_{t_0}^{u \wedge \nu_{nR}} |\langle e_n(s), \bar \sigma^i_n(s)  \rangle |^{\frac{p}{2}} ds \Bigg] \notag
\\
&&\hspace{15mm} =I+II+III+IV \label{main2}
\end{eqnarray}
where $\bar{\bar c}:=\bar{\bar c}(p)$ is the constant of the Burkholder-Davis-Gundy inequality.
\newline
In order to estimate $II$, one observes that the application of A-2 and Cauchy-Schwarz inequality yields,
\begin{eqnarray}
&&\hspace{-12mm} \langle e_n(s), \bar b_n(s) \rangle %&=& \langle X(s)-X_n(s), b(s, X(s))-b_n(s,X_n(\kappa_n(s))) \rangle \notag
%\\
%&=& \langle X(s)-X_n(\kappa_n(s)), b(s, X(s))-b_n(s,X_n(\kappa_n(s))) \rangle \notag
%\\
%&& + \langle X_n(\kappa_n(s))-X_n(s), b(s, X(s))-b_n(s,X_n(\kappa_n(s))) \rangle \notag
%\\
= \langle X(s)-X_n(\kappa_n(s)), b(s, X(s))-b(s,X_n(\kappa_n(s))) \rangle \notag
\\
&& \hspace{10mm}+ \langle X(s)-X_n(\kappa_n(s)), b(s, X_n(\kappa_n(s)))-b_n(s,X_n(\kappa_n(s))) \rangle \notag
\\
&& \hspace{20mm} + \langle X_n(\kappa_n(s)) - X_n(s), b(s, X(s))-b_n(s,X_n(\kappa_n(s))) \rangle \label{impt}
\\
&& \leq  C_R| X(s)-X_n(\kappa_n(s))|^2 +  | X(s)-X_n(\kappa_n(s))|  |b(s, X_n(\kappa_n(s)))-b_n(s,X_n(\kappa_n(s)))| \notag
\\
&& \hspace{15mm} + | X_n(s)-X_n(\kappa_n(s))||b(s, X(s))-b_n(s,X_n(\kappa_n(s))) | \notag
%&& \langle e_n(s), \bar b_n(s) \rangle %&\leq& \bigg(C_R+\frac{1}{2}\bigg)| X(s)-X_n(\kappa_n(s))|^2 +  \frac{1}{2}|b(s, X_n(\kappa_n(s)))-b_n(s,X_n(\kappa_n(s)))|^2 \notag
%\\
%&& \hspace{20mm} + | X_n(s)-X_n(\kappa_n(s))|c\big\{M+M_n+2R\big\}. \notag
%\\
\end{eqnarray}
and on further application of   Young's inequality  and A-1, this becomes
\begin{eqnarray}
\langle e_n(s), \bar b_n(s) \rangle & \leq &  (2C_R+1)| e_n(s)|^2 + (2C_R+1)|X_n(s)-X_n(\kappa_n(s))|^2 \notag
\\
&& \hspace{10mm} + \frac{1}{2}|b(s, X_n(\kappa_n(s)))-b_n(s,X_n(\kappa_n(s)))|^2  \notag
\\
&& \hspace{15mm} + | X_n(s)-X_n(\kappa_n(s))|\big\{c(|X(s)|+|X(\kappa_n(s))|+M+M_n)\big\}. \notag
\end{eqnarray}
Therefore raising power $p/2$ on both sides, we get
\begin{eqnarray}
\hspace{0mm}|\langle e_n(s), \bar b_n(s) \rangle|^{\frac{p}{2}} & \leq & 2^{p-3}\Big\{2(2C_R+1)^{\frac{p}{2}}| e_n(s)|^p + 2 (2C_R+1)^{\frac{p}{2}} |X_n(s)-X_n(\kappa_n(s))|^p \notag
\\
&& \hspace{2mm} +  |b(s, X_n(\kappa_n(s)))-b_n(s,X_n(\kappa_n(s)))|^p  \notag
\\
&& \hspace{10mm} + 2 | X_n(s)-X_n(\kappa_n(s))|^{\frac{p}{2}} \big\{c(|X(s)|+|X(\kappa_n(s))|+M+M_n)\big\}^{\frac{p}{2}} \Big\} \notag
\end{eqnarray}
for every $s \in [t_0, u \wedge \nu_{nR}]$.
Hence one obtains the following estimate on using H\"older's inequality,
\begin{eqnarray}
 \hspace{-8mm} II &:= & (32T)^{\frac{p-2}{2}} \mathbb E\Bigg[\mathbbm 1_{\{C_R \leq f(R)\}} \int_{t_0}^{u \wedge \nu_{nR}}|\langle e_n(s), \bar b_n(s) \rangle |^{\frac{p}{2}} ds \Bigg] \notag
\\
&  \leq &(32T)^{\frac{p-2}{2}}\Bigg\{ 2(2f(R)+1)^{\frac{p}{2}}  \int_{t_0}^{u } \mathbb E\bigg[  \sup_{t_0 \leq r \leq s}| e_n(r \wedge \nu_{nR})|^p \mathbbm 1_{C_R \leq f(R)} \bigg] ds   \notag
\\
&& \hspace{2mm}+ 2 (2f(R)+1)^{\frac{p}{2}} \mathbb E\Bigg[\int_{t_0}^{t_1 } |X_n(s \wedge \nu_{nR})-X_n(\kappa_n(s \wedge \nu_{nR}))|^p ds \Bigg] \notag
\\
&& \hspace{10mm} +   \mathbb E\Bigg[\int_{t_0}^{t_1}|b(s \wedge \nu_{nR}, X_n(\kappa_n(s \wedge \nu_{nR})))-b_n(s\wedge \nu_{nR},X_n(\kappa_n(s\wedge \nu_{nR})))|^pds \Bigg] \notag
\\
&& \hspace{-5mm} + 2^{p+1/2}Tc^{\frac{p}{2}}\sqrt{\mathbb E\Bigg[\int_{t_0}^{t_1 } | X_n(s \wedge \nu_{nR})-X_n(\kappa_n(s \wedge \nu_{nR}))|^{p} ds \Bigg]} \times \sqrt{K+N} \Bigg\}. \label{II2}
\end{eqnarray}
Now in order to estimate $III$ of \eqref{main2}, one observes that for every $i \in \{1, \ldots, m\}$, the assumption A-2 yields
\begin{eqnarray}
|\bar \sigma^i_n(s)|^p  &=&  |\sigma^i(s, X(s))-\sigma_n^i(s, X_n(\kappa_n(s)))|^p \leq   3^{p-1} \Big\{C_R^{\frac{p}{2}} |e_n(s)|^p +  C_R^{\frac{p}{2}} |X_n(s)-X_n(\kappa_n(s))|^p  \notag
\\
&& \hspace{35mm}  +     |\sigma^i(s, X_n(\kappa_n(s)))-\sigma_n^i(s, X_n(\kappa_n(s)))|^p \Big\}\notag
\end{eqnarray}
which implies
\begin{eqnarray}
III &:= &(4mT)^{\frac{p-2}{2}}\sum_{i=1}^{m}\mathbb E \Bigg[\mathbbm 1_{\{C_R \leq f(R)\}}\int_{t_0}^{u \wedge \nu_{nR}} |\bar \sigma^i_n(s)|^p ds \Bigg] \notag
\\
& \leq & 3^{p-1} (4mT)^{\frac{p-2}{2}} \Bigg\{ m f(R)^{\frac{p}{2}} \int_{t_0}^{u } \mathbb E \Bigg[\sup_{t_0 \leq r \leq s}|e_n(r \wedge \nu_{nR})|^p \mathbbm 1_{\{C_R \leq f(R)\}}\Bigg] ds \notag
\\
&& \hspace{40mm} + m  f(R)^{\frac{p}{2}} \mathbb E \Bigg[\int_{t_0}^{t_1} |X_n(s  \wedge \nu_{nR})-X_n(\kappa_n(s \wedge \nu_{nR} ))|^p ds\Bigg]  \notag
\\
&& \hspace{-10mm}+    \sum_{i=1}^{m} \mathbb E \Bigg[ \int_{t_0}^{t_1} |\sigma^i(s \wedge \nu_{nR} , X_n(\kappa_n(s \wedge \nu_{nR})))-\sigma_n^i(s \wedge \nu_{nR}, X_n(\kappa_n(s \wedge \nu_{nR})))|^p ds \Bigg] \Bigg\}.\label{III2}
\end{eqnarray}
Finally, to estimate $IV$, observes that for $i \in \{1, \ldots, m\}$, an application of assumption A-2 yields
\begin{eqnarray}
|\langle e_n(s), \bar \sigma^i_n(s) \rangle |^{\frac{p}{2}} &\leq & |e_n(s)|^{\frac{p}{2}} |\bar \sigma^i_n(s)  |^{\frac{p}{2}}  \notag
= |e_n(s)|^{\frac{p}{2}} |\sigma^i(s, X(s))-\sigma_n^i(s, X_n(\kappa_n(s)))|^{\frac{p}{2}} \notag
%\\
%& \leq &  3^{\frac{p}{2}-1} \Big\{C_R^{\frac{p}{4}} |e_n(s)|^p +  C_R^{\frac{p}{4}} |e_n(s)|^{\frac{p}{2}}|X_n(s)-X_n(\kappa_n(s))|^{\frac{p}{2}}  \notag
%\\
%&& \hspace{10mm}  +     |e_n(s)|^{\frac{p}{2}} |\sigma^i(s, X_n(\kappa_n(s)))-\sigma_n^i(s, X_n(\kappa_n(s)))|^{\frac{p}{2}} \Big\} \notag
%\\
%& \leq &  3^{\frac{p}{2}-1} C_R^{\frac{p}{4}} |e_n(s)|^p + \frac{3^{\frac{p}{2}-1} C_R^{\frac{p}{4}}}{2} |e_n(s)|^p +\frac{3^{\frac{p}{2}-1} C_R^{\frac{p}{4}}}{2}|X_n(s)-X_n(\kappa_n(s))|^p  \notag
%\\
%&& \hspace{10mm}  +    \frac{3^{\frac{p}{2}-1}}{2} |e_n(s)|^p +  \frac{3^{\frac{p}{2}-1}}{2} |\sigma^i(s, X_n(\kappa_n(s)))-\sigma_n^i(s, X_n(\kappa_n(s)))|^p \notag
%\\
\end{eqnarray}
which by using Young's inequality becomes
\begin{eqnarray}
|\langle e_n(s), \bar \sigma^i_n(s) \rangle |^{\frac{p}{2}}  & \leq &  3^{\frac{p-2}{2}} \Big\{C_R^{\frac{p}{4}} |e_n(s)|^p +  C_R^{\frac{p}{4}} |e_n(s)|^{\frac{p}{2}}|X_n(s)-X_n(\kappa_n(s))|^{\frac{p}{2}}  \notag
\\
&& \hspace{10mm}  +     |e_n(s)|^{\frac{p}{2}} |\sigma^i(s, X_n(\kappa_n(s)))-\sigma_n^i(s, X_n(\kappa_n(s)))|^{\frac{p}{2}} \Big\} \notag
\\
& \leq & \frac{3^{\frac{p-2}{2}}}{2} \Big\{ \Big(3 C_R^{\frac{p}{4}} + 1 \Big) |e_n(s)|^p + C_R^{\frac{p}{4}}|X_n(s)-X_n(\kappa_n(s))|^p \notag
\\
&& \hspace{30mm} +   |\sigma^i(s, X_n(\kappa_n(s)))-\sigma_n^i(s, X_n(\kappa_n(s)))|^p \Big\}. \notag
\end{eqnarray}
Thus
\begin{eqnarray}
IV &:= &2^{\frac{3p-4}{2}} m^{\frac{p-2}{2}} T^{\frac{p-4}{4}}\bar {\bar c }\sum_{i=1}^{m} \mathbb E \Bigg[\mathbbm 1_{\{C_R \leq f(R)\}} \int_{t_0}^{u \wedge \nu_{nR}} |\langle e_n(s), \bar \sigma^i_n(s)  \rangle |^{\frac{p}{2}} ds \Bigg] \notag
\\
& \leq & 2^{\frac{3(p-2)}{2}} 3^{\frac{p-2}{2}} m^{\frac{p-2}{2}} T^{\frac{p-4}{4}} \bar {\bar c }  \Bigg\{m\Big(3 f(R)^{\frac{p}{4}} + 1 \Big)  \int_{t_0}^{u } \mathbb E \Bigg[\sup_{t_0 \leq r \leq s}|e_n(r \wedge \nu_{nR} )|^p \mathbbm 1_{\{C_R \leq f(R)\}}  \Bigg] ds \notag
\\
&& \hspace{20mm} + mf(R)^{\frac{p}{4}} \mathbb E \Bigg[\int_{t_0}^{t_1} |X_n(s \wedge \nu_{nR})-X_n(\kappa_n(s \wedge \nu_{nR}))|^p ds \Bigg] \notag
\\
&& \hspace{0mm} +  \sum_{i=1}^{m} \mathbb E \Bigg[\int_{t_0}^{t_1}|\sigma^i(s\wedge \nu_{nR}, X_n(\kappa_n(s \wedge \nu_{nR})))-\sigma_n^i(s \wedge \nu_{nR}, X_n(\kappa_n(s \wedge \nu_{nR})))|^p ds \Bigg]\Bigg\}. \label{IV2}
\end{eqnarray}
Hence using \eqref{II2}, \eqref{III2} and \eqref{IV2} in \eqref{main2} and then after applying Gronwall's inequality, one obtains
%\begin{comment}
%**************************************************************************
%\begin{eqnarray}
%&&  \hspace{-10mm} \mathbb E \Bigg[\sup_{t_0 \leq t \leq u}|e_n(t \wedge \nu_{nR})|^p \mathbbm 1_{\{C_R \leq f(R)\}} \Bigg] = 2^{p-2}\mathbb E\big[|e_n(t_0)|^p\big] \notag
%\\
%&& \hspace{-5mm} + K\big(p,m,T,f(R)\big)\int_{t_0}^{u } \mathbb E \Bigg[\sup_{t_0 \leq r \leq s}|e_n(r \wedge \nu_{nR} )|^p \mathbbm 1_{\{C_R \leq f(R)\}}  \Bigg] ds \notag
%\\
%&&  + K\big(p,m,T,f(R)\big) \mathbb E \Bigg[\int_{t_0}^{t_1} |X_n(s \wedge \nu_{nR})-X_n(\kappa_n(s \wedge \nu_{nR}))|^p ds \Bigg]  \notag
%\\
%&& \hspace{5mm} + K(p, T, f(R)) \mathbb E\Bigg[\int_{t_0}^{t_1}|b(s \wedge \nu_{nR}, X_n(\kappa_n(s \wedge \nu_{nR})))-b_n(s\wedge \nu_{nR},X_n(\kappa_n(s\wedge \nu_{nR})))|^pds \Bigg] \notag
%\\
%&& \hspace{10mm}+ K(p, m, T) \sum_{i=1}^{m}\mathbb E \Bigg[\int_{t_0}^{t_1}|\sigma^i(s\wedge \nu_{nR}, X_n(\kappa_n(s \wedge \nu_{nR})))-\sigma_n^i(s \wedge \nu_{nR}, X_n(\kappa_n(s \wedge \nu_{nR})))|^p ds \Bigg] \notag
%\\
%&& \hspace{15mm}+ K(p,T) \sqrt{\mathbb E\Bigg[\int_{t_0}^{t_1 } | X_n(s \wedge \nu_{nR})-X_n(\kappa_n(s \wedge \nu_{nR}))|^{p} ds \Bigg]} \sqrt{\mathbb E\big[\big\{K_R(1+M_n)\big\}^p \big]}. \notag
%\end{eqnarray}
%where in the above expression, the constants $K$'s are approriately defined.  Therefore using Gronwall's inequality one obtains the following estimate,
%**********************************************************************
%\end{comment}
\begin{eqnarray}
&&  \hspace{-10mm} \mathbb E \Bigg[\sup_{t_0 \leq t \leq t_1}|e_n(t \wedge \nu_{nR})|^p \mathbbm 1_{\{C_R \leq f(R)\}} \Bigg] \leq e^{\bar K_1 T} \Bigg\{2^{p-2}\mathbb E\big[|e_n(t_0)|^p\big] \notag
\\
&&  + \bar K_2 \mathbb E \Bigg[\int_{t_0}^{t_1} |X_n(s \wedge \nu_{nR})-X_n(\kappa_n(s \wedge \nu_{nR}))|^p ds \Bigg]  \notag
\\
&& \hspace{5mm} + \bar K_3 \mathbb E\Bigg[\int_{t_0}^{t_1}|b(s \wedge \nu_{nR}, X_n(\kappa_n(s \wedge \nu_{nR})))-b_n(s\wedge \nu_{nR},X_n(\kappa_n(s\wedge \nu_{nR})))|^pds \Bigg] \notag
\\
&& \hspace{0mm} + \bar K_5 \sum_{i=1}^{m}\mathbb E \Bigg[\int_{t_0}^{t_1}|\sigma^i(s\wedge \nu_{nR}, X_n(\kappa_n(s \wedge \nu_{nR})))-\sigma_n^i(s \wedge \nu_{nR}, X_n(\kappa_n(s \wedge \nu_{nR})))|^p ds \Bigg] \Bigg\} \notag
\\
&& \hspace{10mm}+  \bar K_4 \sqrt{\mathbb E\Bigg[\int_{t_0}^{t_1 } | X_n(s \wedge \nu_{nR})-X_n(\kappa_n(s \wedge \nu_{nR}))|^{p} ds \Bigg]}   \label{*p*1}
\end{eqnarray}
where constants $\bar K_1, \bar K_2, \bar K_3, \bar K_4$ and $\bar K_5$ are appropriately defined and depend explicitly on $f(R)$ but not on $R$ and they are independent of $n$. Thus one can choose $\eta$ sufficiently small and $R$ sufficiently large such that for $\epsilon>0$ (however small),
\begin{eqnarray}
\frac{ \eta p}{q} 2^{q} K < \frac{\epsilon}{3} ,\qquad  \frac{q-p}{q \eta^{p/(q-p)}} \frac{2K}{R^p} < \frac{\epsilon}{3} \qquad \mbox{and} \qquad \frac{q-p}{q \eta^{p/(q-p)}} \mathbb P(C_R > f(R)) < \frac{\epsilon}{3}. \notag
\end{eqnarray}
Therefore, substituting equations \eqref{two2} and \eqref{*p*1} in \eqref{one} and then using assumption  A-3, A-4 and Lemma \ref{l4}, one obtains \eqref{uu} for $p>4$. The application of H\"older's inequality completes the proof.
\end{proof}
\section{Proof of Main Result}
\label{pruf}
In this section, we shall prove Theorem \ref{mainthm} by considering SDDE \eqref{sdde} as a special case of an SDE with random coefficients. In particular, we set
\begin{eqnarray}
b(t,x):=\beta(t, Y(t), x),  \sigma(t,x):=\alpha(t, Y(t), x), b_n(t,x):=\beta(t, Y_n(t), x), \sigma_n(t,x):=\alpha(t, Y_n(t), x) \label{coff}
\end{eqnarray}
for any $t \in [0, T]$ and $x\in \mathbb R^d$.
%\begin{comment}
%*************************************************************************
%First we prove the following lemma.
%\begin{lemma}
%\label{mbd}
%Suppose C-1 holds, then for all $p \geq  1$,
%\begin{eqnarray}
%\mathbb E\Bigg[\sup_{0 \leq t \leq T}|Z(t)|^p \Bigg]  \vee \sup_{n \geq 1}\mathbb E\Bigg[\sup_{0 \leq t \leq T}|X_n(t)|^p \Bigg] < K \notag
%\end{eqnarray}
%for  some $K>0$.
%\end{lemma}
%\begin{proof}
%First we show that
%\begin{eqnarray}
%\mathbb E\Bigg[\sup_{(i-1)\tau \leq t \leq i\tau}|X(t)|^p \Bigg] < K \label{jj}
%\end{eqnarray}
%for every $i \in \{1, \ldots, N\}$ and for all $p \geq 1$. When $t \in [0, \tau]$ , we regard SDDE \eqref{sdde} as an ordinary SDE  and hence one obtains \eqref{jj} from classical result (see  \cite{highammaistaurt} or \cite{mao1997}). For inductive arguments, we assume that \eqref{jj}   holds for $t \in [(r-1)\tau, r \tau]$. Then for $t \in [r\tau, (r+1)\tau]$, we regard SDDE \eqref{sdde} as SDE \eqref{sderc} with $t_0=r \tau$, $t_1=(r+1) \tau$ and coefficients as given in \eqref{coff}. One observes that  A-1 follows from C-1 with $M(t):=|Y(t)|^l $. A-5 follows from inductive assumption and hence \eqref{jj} holds due to Lemma \ref{bbnn}. One can adopt similar arguments to complete the proof with the remark that initial data $\xi$ does not depend on $n$ .
%\end{proof}
%We are now ready to proceed with the proof of our main result of this article i.e. Theorem \ref{mainthm}.
%***************************************************************************************************
%\end{comment}
\begin{proof}[Proof of theorem \ref{mainthm}.]
We shall prove the result using an induction method so as to show that,
\begin{eqnarray}
\lim_{n \rightarrow \infty}\mathbb E\Bigg[\sup_{(i-1)\tau \leq t \leq i \tau} |X(t)-X_n(t)|^p\Bigg] =0 \label{nanu}
\end{eqnarray}
for every $i \in \{1, \ldots, N\}$ and for all $p >0$.
\newline
\textbf{Case: $\mathbf{t \in [0, \tau]}$.} When $t \in [0, \tau]$, then SDDE \eqref{sdde} and Euler  scheme \eqref{sddeem} become  SDE \eqref{sderc} and scheme \eqref{em} with $t_0=0$, $t_1=\tau$, $X(0)=X_n(0)=\xi(0)$ and coefficients given by \eqref{coff}. Also one observes that  A-1 to A-5 hold due to C-1 to C-3. In particular, A-1 is a consequence of C-1 with $M=M_n=1+\Psi^{l}$, where $\Psi:=\sup_{t \in [0, \tau]}|(\xi(\delta_1(t)), \ldots, \xi(\delta_k(t)))|$, which is uniformly bounded due to the fact that $\xi \in C_{\mathcal F_{0}}^b$. Further,  A-2 is a consequence of C-2  since one observes that $C_R$ is deterministic and in fact $C_R=f(R)$, where $f(R)= K_\Psi \mathbbm{1}_{\{0 < R <\Psi\}} + K_R \mathbbm{1}_{\{ R \geq \Psi\}}$ and  therefore $\mathbb P(C_R > f(R))=0$ for any $R>0$. Finally, A-3, A-4 and A-5 hold trivially. Therefore, for $i=1$, equation \eqref{nanu} holds due to Theorem \ref{**} and Lemma \ref{bbnn} for $t_0=0$ and $t_1=\tau$.

For inductive arguments, we assume that when  $i=r$, i.e. $t \in [(r-1)\tau, r\tau]$ for some $r \in \{1, \ldots, N-1\}$,  equation \eqref{nanu} is satisfied and Lemma \ref{bbnn} holds for $t_0=(r-1)\tau$ and $t_1 = r \tau$.  Then we claim that when $i=r+1$, i.e. $t \in [r\tau, (r+1)\tau]$, equation \eqref{nanu}  holds and Lemma \ref{bbnn} is true  for $t_0=r \tau$ and $t_1=(r+1)\tau$.
\newline
\textbf{Case: $\mathbf{t \in [r \tau, (r+1)\tau]}$.} When $t \in [r\tau, (r+1)\tau]$, then SDDE \eqref{sdde} and its Euler  scheme \eqref{sddeem} become SDE \eqref{sderc} and scheme \eqref{em} with $t_0=r\tau$, $t_1=(r+1)\tau$, $X(t_0)=X(r\tau)$, $X_n(t_0)=X_n(r\tau)$ and coefficients given by \eqref{coff}.
\newline
\textit{Verify A-1.} Consider for $t \in [r\tau, (r+1)\tau]$ and $x \in \mathbb R^d$, then due to C-1,
\begin{eqnarray}
|b(t,x)| + |\sigma(t,x)|  = |\beta(t,Y(t),x)| + |\alpha(t,Y(t), x)|\leq  G(M+|x|) \notag
\end{eqnarray}
and
\begin{eqnarray}
|b_n(t,x)| + |\sigma_n(t,x)|= |\beta(t,Y_n(t), x)| + |\alpha(t,Y_n(t), x)| \leq  G(M_n+|x|) \notag
\end{eqnarray}
where $M:=1+\sup_{r\tau \leq t \leq  (r+1)\tau}|Y(t)|^l$ and $M_n:=1+\sup_{r\tau \leq t \leq (r+1)\tau}|Y_n(t)|^l$ which are bounded in $\mathcal L^p$ for every $p >0$ because of Lemma \ref{bbnn} and the inductive assumptions.
\newline
\textit{Verify A-2.} For every $R>0$, $|x|, |z| \leq R$ and $t \in [r\tau, (r+1)\tau]$, from C-2, one obtains,
\begin{eqnarray}
&& \langle x-z, b(t,x)-b(t,z) \rangle = \langle x-z,\beta(t, Y(t), x)-\beta(t, Y(t), z)\rangle \leq  C_R |x-z|^2 \notag
\\
&& \hspace{20mm} |\sigma(t,x)-\sigma(t,z)|^2 = |\alpha(t, Y(t), x)-\alpha(t, Y(t), z)|^2 \leq  C_R |x-z|^2 \notag
\end{eqnarray}
where the random variable  $C_R$ is given  by
\begin{eqnarray}
C_R:=K_R 1_{\Omega_R} + \sum_{j=R}^{\infty} K_{j+1} 1_{\{\Omega_{j+1} \backslash {\Omega}_{j}\}} \notag
\end{eqnarray}
with $\Omega_j:=\{\omega \in \Omega:\sup_{t \in [r \tau, (r+1) \tau]}|Y(t)| \leq j \}$. One observes that $C_R$ is an $\mathcal F_{r\tau}$-measurable random variable. Further one takes $f(R):=K_R$ for every $R >0$ so that
\begin{eqnarray}
\mathbb P(C_R >f(R)) &=& \mathbb P(C_R >K_R) = 1- \mathbb P(C_R  \leq K_R) \notag
\\
&\leq & 1-\mathbb P(\Omega_R)= \mathbb P\Big(\sup_{r \tau \leq t <(r+1)\tau}|Y(t)|>R\Big) \rightarrow 0 \, \, \mbox{as} \, \, R \rightarrow \infty. \label{f*}
\end{eqnarray}
\textit{Verify A-3.}
For $R>0$ and $|x| \leq R$, one observes that $|Y(t)-Y_n(t)| \stackrel{\mathbb P} \rightarrow 0$ as  $\mathbb E[|Y(t)-Y_n(t)|^p] \rightarrow 0$ for every $p >0$ when $n \rightarrow \infty$. Therefore from C-3,
\begin{eqnarray}
\sup_{|x| \leq R}|\beta(t, Y_n(t), x) - \beta(t, Y(t), x) |^p \stackrel{\mathbb P} \rightarrow 0 \quad \mbox{and} \quad \sup_{|x|\leq R}|\alpha(t, Y_n(t), x) -\alpha(t, Y(t), x)|^p \stackrel{\mathbb P} \rightarrow 0.  \notag
\end{eqnarray}
Furthermore, one observes that sequences $$\bigg\{\sup_{|x|\leq R}|\beta(t, Y_n(t), x) - \beta(t, Y(t), x) |^p\bigg \}_{n\geq 1} \, \mbox{and} \,  \bigg \{\sup_{|x|\leq R}|\alpha(t, Y_n(t), x) -\alpha(t, Y(t), x)|^p\bigg \}_{n \geq 1}$$ are uniformly integrable since they are bounded in $\mathcal L^q$ for any $q >1$,
\begin{eqnarray}
\lefteqn{\mathbb E\bigg[\sup_{|x| \leq R}|\beta(t, Y_n(t), x) - \beta(t, Y(t), x) |^{pq} \bigg]=\mathbb E\bigg[\sup_{|x| \leq R} |\beta(t, Y_n(t), x) - \beta(t, Y(t), x) |^{pq} \bigg]} \notag
\\
&\leq & 2^{pq-1} \mathbb E\bigg[\sup_{|x| \leq R}|\beta(t, Y_n(t), x)|^{pq}\bigg] + 2^{pq-1} \mathbb E\bigg[\sup_{|x| \leq R} |\beta(t, Y(t), x) |^{pq} \bigg] \notag
\\
&\leq & 6^{pq-1} G^{pq} \big\{1+ \mathbb E[|Y_n(t)|^{lpq}]+|R|^{pq}\big\}  + 6^{pq-1} G^{pq} \big\{1+ \mathbb E[|Y(t)|^{lpq}]+|R|^{pq}\big\} \notag
\\
&\leq &  6^{pq-1} G^{pq} 2 \big\{1+ K +|R|^{pq}\big\}  \notag
\end{eqnarray}
and similarly, for the sequence $\{\sup_{|x| \leq R}|\alpha(t, Y_n(t), x) -\alpha(t, Y(t), x)|^p\}_{n \geq 1}$. Therefore,
\begin{eqnarray}
\mathbb E\bigg[\sup_{|x| \leq R}\Big\{|\beta(t, Y_n(t), x) - \beta(t, Y(t), x) |^p + |\alpha(t, Y_n(t), x) -\alpha(t, Y(t), x)|^p\Big\}\bigg] \rightarrow 0 \notag
\end{eqnarray}
as $n \rightarrow \infty$ due to Dominated Convergence Theorem which also yields
\begin{eqnarray}
\lim_{n \rightarrow \infty} \mathbb E\Bigg[\int_{r\tau}^{(r+1)\tau}\sup_{|x| \leq R}\{|b(s,x)-b_n(s,x)|^p + |\sigma(s,x)-\sigma_n(s,x)|^p\}ds\Bigg] =0. \notag
\end{eqnarray}
Thus A-3 is satisfied.
\newline
\textit{Verify A-4.} This holds due to the inductive assumptions.
\newline
\textit{Verify A-5.} This follows due to the inductive assumptions.
\newline
Finally Lemma \ref{bbnn} holds for $t_0=r \tau$ and $t_1=(r+1)\tau$ due to the fact that A-1 holds for $t \in [r \tau, (r+1)\tau]$ and $x \in \mathbb R^d$ and A-5 is true for $t_0=r \tau$.
 This completes the proof.
\end{proof}
\section{Rate of Convergence}
In this section, we shall recover the rates of convergence of the Euler scheme \eqref{sddeem} under different set of assumptions. First we state below the relevant assumptions.
\newline
\textbf{C-4.} There exist constants $C>0$ and  $l_1>0$ such that for any $t \in [0, T]$,
\begin{eqnarray}
 \langle x-x^{'}, \beta(t,  y,  x)-\beta(t,  y,  x^{'}) \rangle \vee|\alpha(t, y, x)-\alpha(t, y, x^{'})|^2 & \leq & C| x-x^{'}|^2 \notag
\\
|\beta(t,  y,  x)-\beta(t,  y^{'},  x) |^2 +|\alpha(t, y, x)-\alpha(t, y^{'}, x)|^2 & \leq & C(1+|y|^{l_1}+|y^{'}|^{l_1})| y-y^{'}|^2 \notag
\end{eqnarray}
for all $ x , x^{'} \in \mathbb R^d$ and $y , y^{'} \in \mathbb R^{d \times k}$.
\newline
\textbf{C-5.} There exist constants $C>0$ and  $l_2>0$ such that for any $t \in [0, T]$,
\begin{eqnarray}
|\beta(t,  y,  x)-\beta(t,  y,  x^{'}) |^2 \leq  C( 1+|x|^{l_2}+|x^{'}|^{l_2}) |x-x^{'}|^2 \notag
\end{eqnarray}
for all $ x , x^{'} \in \mathbb R^d$ and $y \in \mathbb R^{d \times k}$.
\newline
In Corollary \ref{pk1}, we prove that under C-1 and C-4, the rate of convergence of Euler scheme \eqref{sddeem} is one-fourth.  On the other hand, in Corollary \ref{pk} we prove that if one makes  assumption C-5 in addition to C-1 and C-4, then the classical rate  of convergence (one-half) can be recovered (see also \cite{baker2000}, \cite{kuchler2000}, \cite{mao2003},  \cite{maosabanis} and \cite{kloeden2007}).
\begin{cor}
\label{pk1}
Suppose C-1 and C-4 hold, then for all $p>0$, the Euler scheme \eqref{sddeem} converges to the true solution \eqref{sdde} in $\mathcal L^p$-sense with rate $1/4$, i.e.

\begin{eqnarray}
\mathbb E \Bigg [\sup_{0 \leq t \leq T}|X(t)-X_n(t)|^p \Bigg] \leq \bar C n^{-\frac{p}{4}} \notag
\end{eqnarray}
where the constant $\bar C>0$ does not depend on $n$.
\end{cor}
\begin{proof}  In equation \eqref{two2}, we choose $\eta=n^{-\frac{p}{2}}$, $R=n^{\frac{q-2p}{2(q-p)}}$, $q>2p \geq 4$, then
\begin{eqnarray}
\frac{\eta p}{q}2^qK= \frac{ p K 2^q}{q} n^{-\frac{p}{2}}, \quad \frac{q-p}{q \eta^{p/(q-p)}} \frac{2K}{R^p}=\frac{2(q-p)K}{q} n^{-\frac{p}{2}}\,\,\mbox{and}\,\, \frac{q-p}{q\eta^{p/(q-p)}}\mathbb P(C_R > f(R)) \equiv 0  \label{f**}
\end{eqnarray}
where we choose $C_R \equiv C$ and $f(R) \equiv C$. Then one can show that,
\begin{eqnarray}
\mathbb E \Bigg [\sup_{(k-1)\tau \leq t \leq k \tau}|X(t)-X_n(t)|^p \Bigg] \leq \bar C n^{-\frac{p}{4}}  \label{kk}
\end{eqnarray}
for each $k \in \{1, \ldots, N\}$.   As seen before, SDDE \eqref{sdde}  corresponds  to the ordinary SDE \eqref{sderc} with $X(0)=X_n(0)=\xi(0)$, $Y(t)=Y_n(t)=\xi(\delta(t))$ for all $t \in [0, \tau]$. Therefore, one obtains the rate as in \eqref{kk} for $k=1$, which is $1/4$ instead of $1/2$, due to the last term of \eqref{II2}, that is also the last term in \eqref{*p*1},  the estimate in \eqref{four4} and the fact that
\begin{eqnarray}
\mathbb E \Bigg[ \int_{0}^{\tau}|b(s \wedge \nu_{nR}, X_n(\kappa_n(s \wedge \nu_{nR})))-b_n(s \wedge \nu_{nR}, X_n(\kappa_n(s \wedge \nu_{nR})))|^pds\Bigg] \equiv 0 \notag
\end{eqnarray}
and
\begin{eqnarray}
\mathbb E \Bigg[ \int_{0}^{\tau}|\sigma^i(s \wedge \nu_{nR}, X_n(\kappa_n(s \wedge \nu_{nR})))-\sigma^i_n(s \wedge \nu_{nR}, X_n(\kappa_n(s \wedge \nu_{nR})))|^pds\Bigg] \equiv 0 \notag
\end{eqnarray}
for every $i \in \{1, \ldots, m\}$.

One then assumes that \eqref{kk} holds for $k=r$, i.e. $t \in [(r-1) \tau, r\tau ] $  for some $r \in \{1, \ldots, N-1\}$,  so as to show that it also holds for $t \in [r \tau, (r+1)\tau ] $. Due to inductive assumption, the initial data of SDE \eqref{sderc} satisfy
\begin{eqnarray}
\mathbb E\Big[|X(r \tau)-X_n(r \tau)|^p\Big] \leq \bar C n^{-\frac{p}{4}}\label{kk1}
\end{eqnarray}
and due to \eqref{coff}, C-4,  Lemma \ref{bbnn}  and  inductive assumptions, one obtains the following estimates,
\begin{eqnarray}
&&\mathbb E \Bigg[ \int_{r \tau}^{(r+1)\tau}|b(s \wedge \nu_{nR}, X_n(\kappa_n(s \wedge \nu_{nR})))-b_n(s \wedge \nu_{nR}, X_n(\kappa_n(s \wedge \nu_{nR})))|^pds\Bigg] \notag
\\
&=& \mathbb E \Bigg[ \int_{r\tau}^{(r+1)\tau}|\beta(s, Y(s), X_n(\kappa_n(s))) -\beta(s, Y_n(s), X_n(\kappa_n(s)))|^p\mathbbm 1_{\{s \leq \nu_{nR}\}}ds\Bigg] \notag
\\
&\leq&  \int_{r\tau}^{(r+1)\tau} \mathbb E \big[(|1+ |Y(s)|^{l_1}+ |Y_n(s)|^{l_2})^{\frac{p}{2}}|Y(s)-Y_n(s)|^p \big] ds\notag
\\
&\leq&  3^{\frac{p-1}{2}} \sqrt{(1+2K)}\int_{r\tau}^{(r+1)\tau} \sqrt{\mathbb E\big[|Y(s)-Y_n(s)|^{2p} \big]} ds\notag
\\
&\leq& \bar C n^{-\frac{p}{2}}. \label{kk2}
\end{eqnarray}
One similarly  estimates
\begin{eqnarray}
\mathbb E \Bigg[ \int_{r \tau}^{(r+1) \tau}|\sigma^i(s \wedge \nu_{nR}, X_n(\kappa_n(s \wedge \nu_{nR})))-\sigma^i_n(s \wedge \nu_{nR}, X_n(\kappa_n(s \wedge \nu_{nR})))|^pds\Bigg] \leq \bar C n^{-\frac{p}{2}} \label{kk3}
\end{eqnarray}
for every $i \in \{1,\ldots,m\}$.
Therefore, \eqref{kk} holds due to \eqref{*p*1} by taking into account \eqref{four4}, \eqref{f**},   \eqref{kk1},  \eqref{kk2} and  \eqref{kk3} for $t \in [r \tau, (r+1)\tau]$ and $p \geq 4$. Then the application of H\"older's inequality  completes the proof.
\end{proof}
\begin{cor}
\label{pk}
Suppose C-1, C-4 and C-5 hold, then for all $p>0$, the Euler scheme \eqref{sddeem} converges to the true solution \eqref{sdde} in $\mathcal L^p$-sense with rate $1/2$, i.e.
\begin{eqnarray}
\mathbb E \Bigg [\sup_{0 \leq t \leq T}|X(t)-X_n(t)|^p \Bigg] \leq \bar C n^{-\frac{p}{2}} \notag
\end{eqnarray}
where the constant $\bar C>0$ does not depend on $n$.
\end{cor}
\begin{proof}
In order to prove this, one observes that the estimates of the rate of convergence of terms of $II$ in \eqref{II2} is improved if C-5 is used along with C-4. For this, one could replace the right hand side of \eqref{impt} by the following,
\begin{eqnarray}
&&\langle e_n(s), \bar b_n(s) \rangle = \langle e_n(s), b(s, X(s))- b(s, X_n(s)) \rangle + \langle e_n(s), b(s, X_n(s))-b(s,X_n(\kappa_n(s))) \rangle \notag
\\
&& \hspace{40mm} + \langle e_n(s), b(s,X_n(\kappa_n(s)))-b_n(s,X_n(\kappa_n(s))) \rangle \notag
\end{eqnarray}
which on using \eqref{coff} becomes
\begin{eqnarray}
&&\langle e_n(s), \bar b_n(s) \rangle = \langle e_n(s), \beta(s,  Y(s),  X(s))- \beta(s,  Y(s), X_n(s)) \rangle \notag
\\
&& \hspace{25mm} + \langle e_n(s), \beta(s,  Y(s), X_n(s))-\beta(s,  Y(s), X_n(\kappa_n(s))) \rangle \notag
\\
&& \hspace{40mm} + \langle e_n(s), \beta(s,  Y(s), X_n(\kappa_n(s)))-\beta(s,  Y_n(s), X_n(\kappa_n(s))) \rangle. \notag
\end{eqnarray}
By using C-4 and Cauchy-Schwartz and Young's inequalities, one obtains
\begin{eqnarray}
&& \langle e_n(s), \bar b_n(s) \rangle \leq  (C+1)| e_n(s)|^2  + \frac{1}{2}| \beta(s,  Y(s), X_n(s))-\beta(s, Y(s), X_n(\kappa_n(s))) |^2 \notag
\\
&& \hspace{40mm} + \frac{1}{2} |\beta(s,  Y(s), X_n(\kappa_n(s)))-\beta(s,  Y_n(s), X_n(\kappa_n(s))) |^2 \notag
\end{eqnarray}
which due to C-4 and C-5 yields
\begin{eqnarray}
&& \langle e_n(s), \bar b_n(s) \rangle \leq  (C+1)| e_n(s)|^2  + \frac{C}{2} \{1+ |X_n(s)|^{l_2}+ |X_n(\kappa_n(s)))|^{l_2} \} |X_n(s)-X_n(\kappa_n(s)))|^2\notag
\\
&& \hspace{40mm} + \frac{C}{2} \{1+|Y(s)|^{l_1}+|Y_n(s)|^{l_1}\} |Y(s)-Y_n(s)|^2. \notag
\end{eqnarray}
Therefore, one estimates $II$ in \eqref{II2} by the following,
\begin{eqnarray}
&&II:=(32T)^{\frac{p-2}{2}}\mathbb E\Bigg[\mathbbm {1}_{\{C_R \leq f(R)\}}\int_{t_0}^{u \wedge \nu_{nR}} |\langle e_n(s), \bar b_n(s) \rangle |^{\frac{p}{2}} ds \Bigg] \notag
\\
&& \leq  (96T)^{\frac{p-2}{2}} \Bigg\{(C+1)^{\frac{p-2}{2}}\int_{r \tau}^{u} \mathbb E\bigg[\sup_{r\tau \leq t \leq s}| e_n(s)|^p \bigg] ds\notag
\\
&& \hspace{10mm} + \bigg(\frac{C}{2}\bigg)^{\frac{p}{2}} \mathbb E\Bigg[\int_{r \tau}^{u}\{1+ |X_n(s)|^{l_2}+ |X_n(\kappa_n(s)))|^{l_2} \}^{\frac{p}{2}} |X_n(s)-X_n(\kappa_n(s)))|^p ds\Bigg]\notag
\\
&& \hspace{20mm} + \bigg(\frac{C}{2}\bigg)^{\frac{p}{2}} \mathbb E\Bigg[\int_{r \tau}^{u}  \{1+|Y(s)|^{l_1}+|Y_n(s)|^{l_1}\}^{\frac{p}{2}} |Y(s)-Y_n(s)|^p ds\Bigg]\Bigg\} \notag
\end{eqnarray}
which on the application of H\"older's inequality and Lemma \ref{bbnn} yields
\begin{eqnarray}
&&\hspace{-5mm}II \leq  (96T)^{\frac{p-2}{2}} \Bigg\{(C+1)^{\frac{p-2}{2}}\int_{r \tau}^{u} \mathbb E\bigg[\sup_{r\tau \leq t \leq s}| e_n(s)|^p \bigg] ds \notag
\\
&& \hspace{0mm} + \bigg(\frac{C}{2}\bigg)^{\frac{p}{2}} \hat K \int_{r \tau}^{(r+1) \tau} \Bigg[\sqrt{\mathbb E\Big[|X_n(s)-X_n(\kappa_n(s)))|^{2p}\Big]}  +  \sqrt{\mathbb E\Big[ |Y(s)-Y_n(s)|^{2p} \Big]} \Bigg]ds\Bigg\}. \label{above}
\end{eqnarray}
where $\hat K$ is a constant which does not depend on $n$. Therefore, in  \eqref{main2}, one replaces the estimate \eqref{II2} of $II$ by the estimate \eqref{above} whereas $III$ and $IV$ remain the same. Then the first term of \eqref{above} is incorporated with similar terms of $III$ and $IV$ so as to apply Gronwall's inequality. Furthermore, one observes that the second and third terms of \eqref{above} are of  (improved) order $\mathcal O(n^{-p/2})$ due to Lemma \ref{l4} and inductive assumptions respectively. Thus, by adopting same  arguments as adopted in  the proof of Corollary \ref{pk1} with $\eta=n^{-\frac{p}{2}}$, one obtains the desired rate ($1/2$).
\end{proof}

\section{Numerical Examples}
In this section, we shall illustrate our findings with the help of numerical examples. We consider the following SDDE given by

\begin{eqnarray}
dZ(t)=[a Z(t)+ b \{Z(t-\tau)\}^{l_1}]dt + [\beta_1 + \beta_2 Z(t)+ \beta_3 \{Z(t-\tau)\}^{l_2}]dW(t) \quad \mbox{for} \quad t \in [0,2\tau] \label{sddep1}
\end{eqnarray}
with initial data $\xi(t)$ when $t \in [-\tau,0]$,  $\tau >0$ is a fixed delay, $l_1, l_2 >0$, $a, b, \beta_1, \beta_2, \beta_3 \in \mathbb R$.

When $t \in [0, \tau]$, then  explicit solution of SDE \eqref{sddep1} is given by,
\begin{eqnarray}
\lefteqn{Z(t)=\Phi_{0, t} \Big\{ \xi(0) + \int_{0}^{t} \Phi_{0, s}^{-1} \{b \{\xi(s-\tau)\}^{l_1}-\beta_2 (\beta_1+\beta_3 \{\xi(s-\tau)\}^{l_2}) \}ds} \notag
\\
&& \hspace{70mm}+ \int_{0}^{t} \Phi_{0,s}^{-1} \{\beta_1+\beta_3 \{\xi(s-\tau)\}^{l_2}\} dW(s)\Big\}, \label{solp1}
\end{eqnarray}
where,
\begin{eqnarray}
\Phi_{0, t}=\exp\Big(\Big\{a-\frac{\beta_2^2}{2}\Big\} t + \beta_2 W(t) \Big). \notag
\end{eqnarray}
And when $t \in [\tau, 2\tau]$, then the explicit solution  is given by
\begin{eqnarray}
\lefteqn{Z(t)=\Phi_{\tau, t}\Bigg\{ Z(\tau)+ \int_{\tau}^{t} \Phi_{\tau, s}^{-1}[b\{Z(s-\tau)\}^{l_1}-\beta_2(\beta_1+\beta_3\{Z(s-\tau)\}^{l_2})]ds} \notag
\\
&&\hspace{70mm}+ \int_{\tau}^{t} \Phi_{\tau, s}^{-1}(\beta_1+\beta_3\{Z(s-\tau)\}^{l_2})dW(s)\Bigg\} \label{solp2}
\end{eqnarray}
where,
\begin{eqnarray}
\Phi_{\tau, t}=\exp\Big(\Big\{a-\frac{\beta_2^2}{2}\Big\} (t-\tau) + \beta_2 (W(t) -W(\tau)\Big). \notag
\end{eqnarray}
For Euler  scheme, we divide the interval $[0, \tau]$ into  sub-intervals of width $\frac{\tau}{2^N}$ for some positive integer $N$. Then solution  at $kh$-th grid is given by
\begin{eqnarray}
Z_n((k+1)h)&=&[a Z_n(kh)+ b \{\xi(kh-\tau)\}^{l_1}]h \notag
\\
&&\hspace{20mm} + \Big[\beta_1 + \beta_2 Z_n(kh)+ \beta_3 \{\xi(kh-\tau)\}^{l_2}\Big] (W((k+1)h)-W(kh)) \notag
\end{eqnarray}
with $\xi(t)$ for $t \in [-\tau,0]$, $k=1,\ldots, 2^{N}$. We use this solution to find the solution of the Euler  scheme in the interval $[\tau, 2\tau]$.

We take values of the parameters of SDDE \eqref{sddep1} given in Table \ref{tablep1}. From Table \ref{table2}, we observe that as step size is decreased, the error is decreased too. In Figure \ref{figrr}, the reference line has a slope of $-0.5$.
\begin{table}[ht]
\centering
\caption{Values of Parameters}
\begin{tabular}{|c|c|c|c|c|c|c|c|}
\hline
$p$ & $\tau$ & $a$ & $b$ & $\beta_1$ & $\beta_2$ & $\beta_3$ & $\xi(t)$  for $t \in [-\tau, 0]$\\
\hline
$2$ & $1$ & $-8$ & $4$ & $0$ & $1$ & $1$ & t+1\\
\hline
\end{tabular}
\label{tablep1}
\end{table}
\begin{table}[ht]
\centering
\caption{Effect of decreasing step size on Error}
\begin{tabular}{|c|c|c|c|}
\hline
 & \multicolumn{3}{c|}{$\sqrt \mathbb E(|Z(T)-Z_n(T)|^2)$} \\
\cline{2-4}
$h$  & $l_1=1/2, l_2=1/2$ & $l_1=1, l_2=1$  & $l_1=2, l_2=3$ \\
\hline
$2^{-9}$  & $0.0002332678715832590$ & $0.0001857299504105190$  & $0.0002812168347124360$ \\
$2^{-10}$ & $0.0001171164266162860$ & $0.0000911946799085898$  & $0.0001308779599825470$ \\
$2^{-11}$ & $0.0000605972526234064$ & $0.0000450951550271487$  & $0.0000652194083354262$ \\
$2^{-12}$ & $0.0000314672425957171$ & $0.0000223376219662284$  & $0.0000321943295689175$ \\
$2^{-13}$ & $0.0000159827178333670$ & $0.0000111954499489959$  & $0.0000156672183324178$ \\
$2^{-14}$ & $0.0000079146081811971$ & $0.0000052537027007630$  & $0.0000057778221373887$ \\
\hline slope & & $-0.51157$ &  $ -0.54615 $ \\
\hline
\end{tabular}
\label{table2}
\end{table}
%\begin{comment}
\begin{figure}[H]
%\centering
\caption{Strong order of convergence of the Euler  scheme.}
\centering
	\subfloat[$l_1=1, l_2=1$]{
			\includegraphics[width=.45\textwidth]{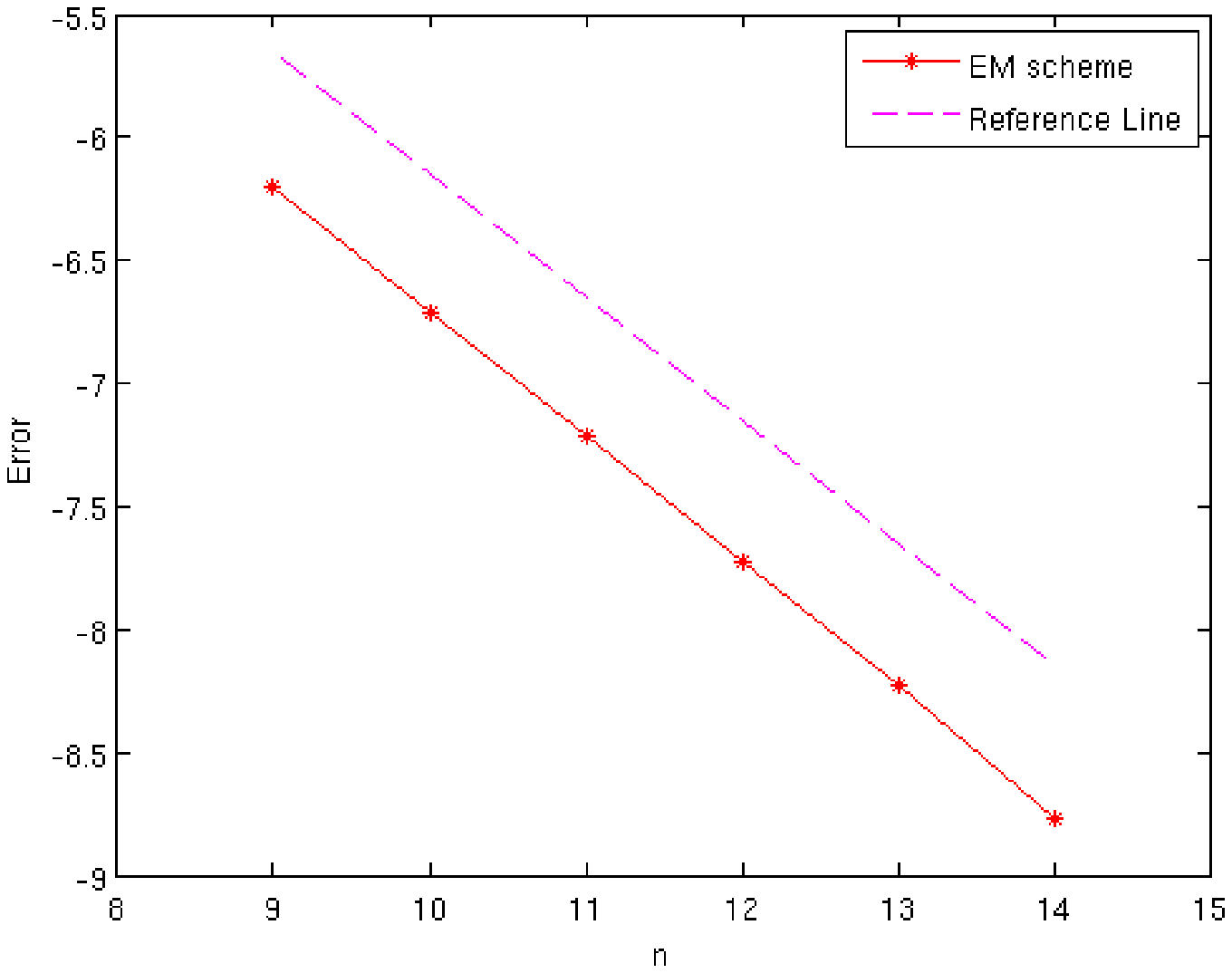}
	}
	%\subfloat[$l_1=1/2, l_2=1/2$]{
	%		\includegraphics[width=.5\textwidth]{l1_l2_half.png}
	%} \\
	\subfloat[$l_1=2, l_2=3$]{
			\includegraphics[width=.45\textwidth]{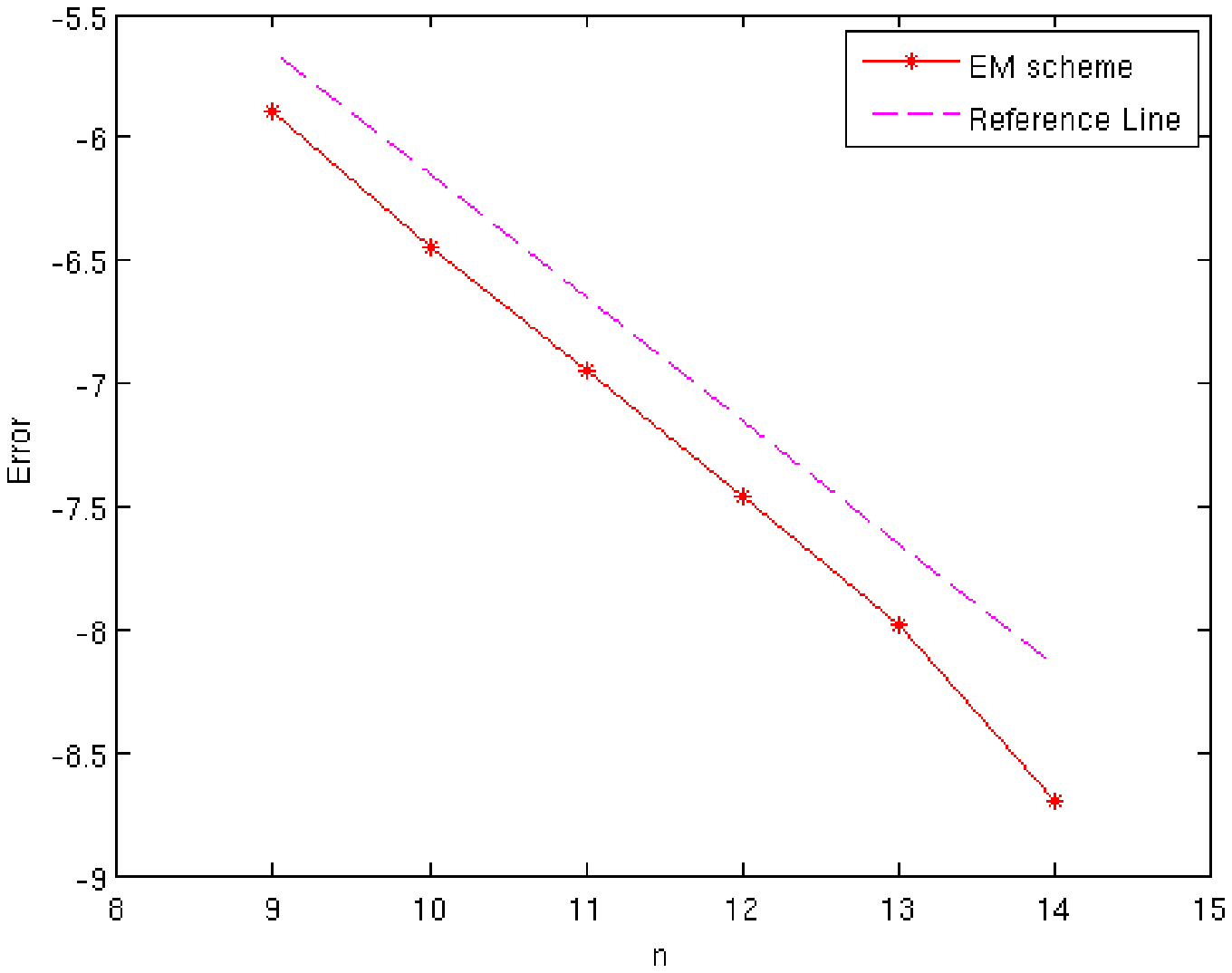}
	}
\label{figrr}
\end{figure}
%\end{comment}
%\newpage
%\bibliographystyle{unsrt}
%\bibliography{refer_new}

\end{document}